\documentclass{article}
\usepackage[utf8]{inputenc}
\usepackage[english]{babel}
\usepackage{pdfpages}
\usepackage[margin=3cm]{geometry}
\usepackage{graphicx}
\usepackage{amsmath,amsthm,amssymb}
\usepackage{mathtools,enumerate,float}
\usepackage{wasysym}
\usepackage{microtype}
\usepackage{hyperref}
\usepackage[capitalise]{cleveref}

\usepackage[draft,inline,nomargin,index]{fixme}

\fxsetup{theme=color,mode=multiuser}
\FXRegisterAuthor{flm}{aflm}{\color{blue}FLM}
\FXRegisterAuthor{rt}{art}{\color{red}RT}

\newenvironment{cproof}{\begin{proof}[Proof of the claim]}{\end{proof}}

\title{On the stability of hyperbolicity under quantitative measure equivalence}
\author{Thiebout Delabie\footnote{supported by grant P2NEP2 181564 of the Swiss National Science Foundation}, Juhani Koivisto, François Le Maître\footnote{Research 
partially supported by Projet ANR-17-CE40-0026 AGRUME, Projet 
ANR-19-CE40-0008 AODynG and IdEx Université de Paris, ANR-18-IDEX-0001.}  and Romain Tessera}
\date{\today}

\newtheorem{theorem}{Theorem}[section]
\newtheorem{lemma}[theorem]{Lemma}
\newtheorem{proposition}[theorem]{Proposition}
\newtheorem{corollary}[theorem]{Corollary}
\newtheorem*{theoremStatement}{Theorem}

\newtheorem{thmintro}{Theorem}
\newtheorem{corintro}[thmintro]{Corollary}

\theoremstyle{definition}

\newtheorem{definition}[theorem]{Definition}
\newtheorem{remark}[theorem]{Remark}
\newtheorem{clai}[theorem]{Claim}

\newtheorem{que}[theorem]{Question}

\newcommand{\abs}[1]{\left\lvert #1\right\rvert}

  \newcommand{\la}{\left\langle}
\newcommand{\ra}{\right\rangle}

\newcommand{\vol}{\operatorname{Vol}}
\newcommand{\ent}{\operatorname{Ent}}
\renewcommand{\leq}{\leqslant}
\renewcommand{\geq}{\geqslant}
\renewcommand{\le}{\leqslant}
\renewcommand{\ge}{\geqslant}
\renewcommand{\epsilon}{\varepsilon}

\newcommand{\Aut}{\operatorname{Aut}}

\newcommand{\SL}{\operatorname{SL}}

\newcommand{\N}{\mathbb{N}}
\newcommand{\Z}{\mathbb{Z}}

\newcommand{\R}{\mathbb{R}}

\newcommand{\eps}{\varepsilon}
\newcommand{\LL}{\mathrm{L}}
\newcommand{\act}{\curvearrowright}
\newcommand{\inv}{^{-1}}

\newcommand{\actom}{*}
\newcommand{\actx}{\cdot}

\makeatletter
\newcommand{\AlignRight}[1]{\ifmeasuring@#1\else\omit\hfill$\displaystyle#1$\fi\ignorespaces}
\makeatother

\begin{document}

\maketitle

\begin{abstract}
A well-known result of Shalom says that lattices in SO$(n,1)$ are $\LL^p$ measure equivalent for all $p<n-1$. 
His proof actually yields the following stronger statement: 
the natural coupling resulting from a suitable choice of fundamental domains from a uniform lattice to a non-uniform one is  $(\LL^p,\LL^{\infty})$.  
Moreover, it is easy to see that the coupling is cobounded: the fundamental domain of the uniform lattice is contained in a union of finitely many translates of the fundamental domain of the non-uniform one.  
The purpose of this note is to prove that this statement is sharp in the following sense: 
if a ME-coupling from a hyperbolic group to a non-hyperbolic group is cobounded and $(\LL^p,\LL^{\infty})$, 
then $p$ must be less than some $p_0$ only depending on the hyperbolic group. \end{abstract}
\tableofcontents
\section{Introduction}

A \textbf{measure equivalence coupling} from a countable group $\Gamma$ to another countable group $\Lambda$ is given by the following data:
\begin{enumerate}[(1)]
    \item a $\sigma$-finite standard measured space $(\Omega,\mu)$ endowed with commuting $\Gamma$ and $\Lambda$ actions denoted by $(\gamma,x)\mapsto \gamma\actom x$ and $(\lambda,x)\mapsto \lambda\actom x$ and
    \item  chosen
\emph{finite measure} \textbf{fundamental domains} $X_\Gamma$ for the $\Gamma$-action and $X_\Lambda$ for the $\Lambda$-action,
\end{enumerate} 
where a fundamental domain (for a given \(\Gamma\)-action on \((\Omega,\mu)\) ) is a measurable subset $X_\Gamma\subseteq X$ 
such that the map $\Gamma\times X_\Gamma\to \Omega$ taking $(\gamma,x)$ to $\gamma\actom x$ is a measure-preserving bijection when endowing $\Gamma\times X_\Gamma$ with the product of the counting measure and the restriction of $\mu$ to $X_\Gamma$.
In other words, both the $\Gamma$ and $\Lambda$ actions on $\Omega$ are measurably trivial free actions, but they are paired together in an interesting way. 
When there is a measure equivalence coupling from $\Gamma$ to $\Lambda$, they are said to be \textbf{measure equivalent} (see \cite[Sec.~2]{furmanGromovMeasureEquivalence1999} for the fact that measure equivalence is an equivalence relation, thus justifying its name).

The most basic example of a measure equivalence coupling is provided by two finite index subgroups of the same group, one acting on by right translations and the other acting by left translations: the fundamental domains are obtained by choosing coset representatives.
More generally, any two lattices in a common locally compact group are measure equivalent.
While measure equivalence was introduced by Gromov as a measured version of quasi-isometry 
\cite[Sec.~0.5.E.]{gromovAsymptoticInvariantsInfinite1993},
it behaves very differently.
For instance, since any orbit equivalent groups are measure equivalent, the Ornstein-Weiss theorem guarantees that all countable infinite amenable groups are pairwise measure equivalent \cite{ornsteinErgodicTheoryAmenable1980}, while finitely generated amenable groups can be distinguished up to quasi-isometry by various invariants (such as volume growth, isoperimetric profile, etc.).

Another good example is provided by hyperbolicity, a cornerstone of geometric group theory: 
uniform lattices in rank one simple Lie groups are hyperbolic but non-uniform ones are not hyperbolic, 
except for SL$(2,\R)$ where
all lattices are hyperbolic\footnote{This is a classical result: see for instance \cite[Example~12.9]{drutuGeometricGroupTheory2018}.}.
Hence these are interesting instances of groups that are measure equivalent, yet with very different geometric properties. 

In this note, we are interested in quantitative refinements of measure equivalence ensuring that hyperbolicity is preserved. 
While the general idea of adding quantitative constrainsts can be traced back to the rigidity result of Belinskaya
\cite{belinskayaPartitionsLebesguespace1968} for $\LL^1$ orbit equivalent $\Z$-actions, it only reemerged a decade ago in the work of Bader, Furman, Sauer \cite{baderIntegrableMeasureEquivalence2013}, and also Austin and Bowen \cite{austinIntegrableMeasureEquivalence2016}.  
A more systematic theory of quantitative measure equivalence have been developped in \cite{delabieQuantitativeMeasureEquivalence2022}. 

Quantitative measure equivalence relies crucially on the fact that any measure equivalence coupling comes with two \emph{cocycles} $\alpha:\Gamma\times X_\Lambda\to \Lambda$ and $\beta:\Lambda\times X_\Gamma\to \Gamma$ where for all $\gamma\in\Gamma$ and $x\in X_\Lambda$, we define  $\alpha(\gamma,x)\in \Lambda$ as the unique element which takes the point $\gamma\actom x\in \Omega$ back into the fundamental domain $X_\Lambda$, namely it is uniquely defined by the formula
\(\alpha(\gamma,x)\actom  (\gamma\actom x)\in X_\Lambda\)
(and vice versa for $\beta$).
 \begin{definition}\label{defIntro:QuantME}
     Let $\phi, \psi \colon \mathbb{R}_{\geq 0}\to \mathbb{R}_{\geq 0}$ be non-decreasing functions and let $S_\Gamma\subset \Gamma$, $S_\Lambda\subset \Lambda$ be finite generating sets for $\Gamma$ and $\Lambda$.  
     We will say that the measure equivalence coupling from $\Gamma$ to $\Lambda$ is \textbf{$(\phi,\psi)$-integrable} if there is $\delta>0$ such that for all $s\in S_\Gamma$ 
     \[
        \int_{X_\Lambda} \phi\left(\delta|\alpha(s,x)|_{S_\Lambda}\right)\, d\mu(x) < +\infty.
     \]
     and for all $t\in S_\Lambda$,
     \[
        \int_{X_\Gamma} \psi\left(\delta|\beta(t,x)|_{S_\Gamma}\right)\, d\mu(x) < +\infty.        
     \]
\end{definition}

Note that we need  the multiplicative constant $\delta$ so that the definition does not depend on the choice of the generating set in general. Such a constant can be dropped when $\phi(t)=t^p$, but not when 
$\phi(t)=\exp(t)$ (see also Remark~\ref{rmk: equivalence vs similarity}). We also refer the reader to Section \ref{sec: rel between fund dom} for another cocycle-free way 
of defining the above notions which was used as a starting point in \cite[Sec.~2]{delabieQuantitativeMeasureEquivalence2022}.

When \(\varphi(t)=t^p\)   (resp.~$\psi(t)=t^p$) we replace \(\varphi\)  (resp.~\(\psi\)) by \(\LL^p\) in the notation for integrability since we are asking that the corresponding cocycle is in \(\LL^p\). Being in $\LL^p$ is itself an integrability condition,
so we will simply speak of $(\LL^p,\psi)$ measure equivalence couplings. This notation is naturally extended to the case \(p=\infty\) (boundedness condition)
or \(p=0\) (no integrability condition):
for instance a measure equivalence coupling is \((\LL^0,\LL^\infty)\) simply when for every \(t\in S_\Lambda\),
the map $\beta(\lambda,\cdot)$ is essentially bounded. 
In addition, an $(\LL^p,\LL^p)$ measure equivalence coupling is commonly referred to as an $\LL^p$ measure equivalence coupling.

It is important to note that unlike measure equivalence, the above integrability conditions do depend on the choice of the fundamental domains $X_\Gamma$ and \(X_\Lambda\):
other such domains may lead to better or worse integrability properties.
However, these integrability conditions still miss an important aspect of the coupling, namely the relative position of the fundamental domains $X_\Gamma$ and $X_\Lambda$. 
So let us say that a measure equivalence coupling from $\Gamma$ to $\Lambda$ is {\bf cobounded} if  $X_{\Gamma}$ is contained in the union of finitely many $\Lambda$-translates of  $X_{\Lambda}$.  
If $X_{\Lambda}$ is also contained in the union of finitely many $\Gamma$-translates of  $X_{\Gamma}$, then we shall say that the coupling is \textbf{mutually cobounded}. 
A particular instance when this property is satisfied is when $X_\Gamma=X_\Lambda$, in which case the coupling is called an orbit equivalence (OE) coupling\footnote{see \cite[Sec.~2.6]{delabieQuantitativeMeasureEquivalence2022} for the connection with the usual notion of orbit equivalence.}.  
To illustrate the relevance of this notion, 
we quote a result of Sauer in his PhD thesis \cite{sauerMathrmLInvariantsGroups2002} (very close to Shalom's \cite[Thm.~2.1.7.]{shalomHarmonicAnalysisCohomology2004}): any two amenable groups are quasi-isometric if and only if there exists an $\LL^{\infty}$  mutually cobounded measure equivalence coupling between them.

\paragraph{Uniform versus non-uniform lattices in rank one simple Lie groups.}

In \cite{shalomRigidityUnitaryRepresentations2000}, Shalom proves that any two lattices in SO$(n,1)$ are $\LL^p$ measure equivalent for all $p<n-1$. When one of the two lattices is uniform, one can strengthen this statement as follows. 
Let $n\geq 2$, and let $\Gamma$ and $\Lambda$ be two lattices  in SO$(n,1)$  such that $\Gamma$ is uniform. 
We consider the coupling associated to the action of $\Gamma$ and $\Lambda$ respectively by left and right-translations on the measure space SO$(n,1)$ equipped with an invariant Haar measure. Shalom showed that for any relatively compact fundamental domain $X_{\Gamma}$ for $\Gamma$ and a suitable fundamental domain $X_{\Lambda}$ for $\Lambda$, the resulting coupling is an $(\LL^p,\LL^{\infty})$ measure equivalence coupling from $\Gamma$ to $\Lambda$, for all $p<n-1$. Exploiting the relative compactness of $X_{\Gamma}$, one can check that this coupling is moreover cobounded.
We summarize this as follows.

\begin{theoremStatement}[{Shalom \cite[Thm.~3.6]{shalomRigidityUnitaryRepresentations2000}}] 
Let $\Gamma$ and $\Lambda$ be two lattices in SO$(n,1)$ such that $\Gamma$ is uniform. Then  there exists a cobounded measure equivalence coupling from $\Gamma$ to $\Lambda$  that is $(\LL^p,\LL^{\infty})$-integrable for all $p<n-1$. 
\end{theoremStatement}
By contrast we prove the following rigidity result.  

\begin{thmintro}[{see Cor.~\ref{cor:Lp-infty}}]\label{thmintroHyper}
	Let $\Gamma$ be a finitely generated hyperbolic group. There 
	exists $p>0$   such that if there exists a 
	cobounded $(\LL^p,\LL^{\infty})$ measure equivalence coupling from a finitely generated group $\Gamma$ to $\Lambda$, then $\Lambda$ is also hyperbolic. 
\end{thmintro}

\begin{remark}
The value of $p$ for which the conclusion holds is explicit: assuming that $\Gamma$ admits a Cayley graph that is $\delta$-hyperbolic and has volume entropy at most $\alpha$, one can take $p= 
	108\delta \alpha+2$. For the definition of volume entropy, see the paragraph which precedes Theorem~\ref{thm:hyperbolicExpME}.
\end{remark}

We immediately deduce the following converse of Shalom's result.
\begin{corintro}\label{corintroLatticerank1LpLinfty}
	Assume that $\Gamma$ is a uniform lattice in a center-free, real rank 1 simple Lie group $G$ and $\Lambda$ is another lattice of $G$. There exists $p$ only depending on $G$ such that if there exists a cobounded $(\LL^p,\LL^{\infty})$ measure equivalence coupling from $\Lambda$ to $\Gamma$, then $\Lambda$ must be uniform as well. \end{corintro}

This raises the following question.
\begin{que}What is the infimum over all $p$ such that the previous result holds? Does it match Shalom's value $n-1$ for lattices in SO$(n,1)$?
\end{que}

As observed by Mikael de la Salle, Corollary~\ref{corintroLatticerank1LpLinfty} is in sharp contrast with what happens for lattices in higher rank simple Lie groups: indeed if  $\Gamma$ and  $\Lambda$ are lattices in a simple Lie group $\Gamma$ of rank $\geq 2$, then if $X_{\Lambda}$ and $X_{\Gamma}$ are  Dirichlet fundamental domains for $\Lambda$ and $\Gamma$, the resulting measure equivalence coupling is exponentially integrable \cite[Lemme
5.6]{delasalleStrongPropertyHigherrank2019}. Hence the following holds: 
\begin{theoremStatement}[de la Salle] 
Let $\Gamma$ and $\Lambda$ be two lattices in a simple Lie group of rank at least $2$, such that $\Gamma$ is uniform. Then there exists a cobounded  $(\exp,\LL^{\infty})$ measure equivalence coupling from $\Gamma$ to $\Lambda$. 
\end{theoremStatement}

Coming back to Theorem~\ref{thmintroHyper}, observe that the $\LL^{\infty}$ condition from $\Lambda$ to $\Gamma$ is the strongest possible. 
 It can be relaxed to an $\LL^p$-type condition, but at the cost of imposing a stretched exponential integrability condition in the other direction. More precisely, we obtain:

\begin{thmintro}[{see Cor.~\ref{cor: phi coupling hyp}}]\label{thmintroHyper2}
		Let $\Gamma$ be a finitely generated hyperbolic group. For every $p>q>0$, if there is a cobounded $(\varphi,\psi)$-integrable measure equivalence coupling from $\Gamma$ to a finitely generated group $\Lambda$, where 
	$\varphi(t)=\exp(t^p)$ and $\psi(t)=t^{1+1/q}$, then $\Lambda$ is also hyperbolic.  
\end{thmintro}
\begin{remark}
This shows for instance that there does not exist $(\exp,\LL^{2+\eps})$ 
measure equivalence couplings from a hyperbolic group to a non hyperbolic group for any $\eps>0$. 
\end{remark}
Once again we deduce the following corollary for lattices in rank 1 simple Lie groups, which again contrasts with the the case of higher rank lattices.

\begin{corintro}\label{corintroLatticerank1}
	Assume that $\Gamma$ is a uniform lattice in a center-free, real rank 1 simple Lie group $G$ and $\Lambda$ is another lattice of $G$.  For every $p>q>0$, if there is a cobounded $(\varphi,\psi)$-integrable measure equivalence coupling from $\Gamma$ to $\Lambda$ where 
	$\varphi(t)=\exp(t^p)$ and $\psi(t)=t^{1+1/q}$, then $\Lambda$  is uniform as well. \end{corintro}

\begin{remark}
In Corollaries~\ref{corintroLatticerank1LpLinfty} and~\ref{corintroLatticerank1}, the case of  $\SL(2,\R)$ was already known and actually a much stronger conclusion holds in that case: Bader, Furman and Sauer have proved that non-uniform lattices and uniform ones are not $\LL^1$ measure equivalent.   
\end{remark}
\begin{remark}
Theorems~\ref{thmintroHyper} and~\ref{thmintroHyper2} should be compared with a theorem of Bowen saying that if there exists an $(\LL^0, \LL^1)$ orbit equivalence coupling from a virtually free group to a finitely generated accessible group $\Lambda$, then $\Lambda$ is virtually free \cite{Bowen2017}. 
\end{remark}

\paragraph{Stability of hyperbolicity under quantitative orbit equivalence}
From an orbit equivalence coupling, a measure equivalence coupling can be constructed in such a way that both groups share a common fundamental domain. In particular, such a coupling is automatically cobounded. 
Hence \cref{thmintroHyper} and \cref{thmintroHyper2} have the following immediate corollaries. 

\begin{corintro}\label{corintro: stability of hyp under LpLinf}
 Let $\Gamma$ be a finitely generated hyperbolic group. There 
	is $p>0$   such that if there exists a 
$(\LL^p,\LL^{\infty})$ orbit equivalence coupling from $\Gamma$ a finitely generated group $\Lambda$, then $\Lambda$ is also hyperbolic. 
 \end{corintro}

\begin{corintro}
Let $\Gamma$ be a finitely generated hyperbolic group. For every $p>q>0$, if there is a $(\varphi,\psi)$-integrable orbit equivalence coupling from $\Gamma$ to  a finitely generated group $\Lambda$ where 
	$\varphi(t)=\exp(t^p)$ and $\psi(t)=t^{1+1/q}$, then $\Lambda$ is also hyperbolic. 
\end{corintro}


\paragraph{Hyperbolicity and embedded cycles.}

The proofs of Theorem~\ref{thmintroHyper}  and Theorem~\ref{thmintroHyper2} are based 
on the following characterization of {\it non}-hyperbolic spaces. We denote by $C_n$ the cyclic graph of length $n$.

\begin{theoremStatement}[{\cite[Prop.~5.1]{humePoorlyConnectedGroups2020}}]\label{thmIntro:Nonhyperbolic}
	Let $X$ be a connected graph that is not hyperbolic. Then for all $n\in \N$, there exists a 18-bi-Lipschitz embedded cyclic
subgraph in $X$ of length at least $n$.\end{theoremStatement}

The strategy of proof consists in confronting this result with the following one (a very close statement is proved in~\cite{verbeekMetricEmbeddingHyperbolic2014} for the real hyperbolic space).

\begin{thmintro}[see Corollary~\ref{cor:HypNotEmbed}]\label{thmIntro:NonhyperbolicQuant}
	Let $a\geq 0$, $b\geq 1$, and $\delta\geq 1$. There is an integer $n_0=n_0(a,b)\geq 2$ such that the following holds. For all  $\delta$-hyperbolic geodesic space $X$, for all $n\geq n_0$, if there is a  map $\varphi\colon C_n\to X$ such that for all $x,y\in C_n$
	\[ad_{C_n}(x,y) \leq d(\varphi(x),\varphi(y)) \leq bd_{C_n}(x,y)\]
	then we have 
\begin{equation}\label{eq:hypcycle'}
a< 6\delta\cdot  \frac{\log n}{n}.
\end{equation}
\end{thmintro}

\paragraph{Sketch of proof.}

We prove Theorem~\ref{thmintroHyper} and Theorem~\ref{thmintroHyper2} by contradiction. Let us briefly sketch the argument for an orbit equivalence coupling. We assume that $\Lambda$ is non hyperbolic and consider  the map  $\varphi_n\colon C_n\to \Lambda$ provided for some large integer $n$ by Theorem~\ref{thmIntro:Nonhyperbolic}. Identifying the orbits of $\Lambda$ and $\Gamma$, we obtain for a.e.\ $x\in X$, a map $\psi_{n,x}\colon C_n\to \Gamma$. By exploiting the integrability condition from $\Lambda$ to $\Gamma$, we obtain a bound on the Lipschitz constant of $\psi_{n,x}$, which is satisfied on a subset of $X$ of sufficiently large measure. Observe that this step is trivial under the hypotheses of Theorem~\ref{thmintroHyper}, as the $L^{\infty}$-condition ensures that $\psi_{n,x}$ is Lipschitz,  uniformly with respect to $x\in X$. 
The more subtle part of the argument consists in estimating the Lipschitz constant of the inverse of $\psi_{n,x}$, in order to obtain a contradiction with Theorem~\ref{thmIntro:NonhyperbolicQuant}.

In case of a measure equivalence coupling, a difficulty arises in the second step of the proof: in order to estimate the Lipschitz constant of the inverse of $\psi_{n,x}$, 
we need to go back to $\Lambda$ and exploit the lower bound $\frac{1}{18}$ on the Lipschitz constant of the inverse of $\varphi_{n}$. This where the coboundedness condition comes in, allowing us to relate
 the fundamental domains of the two groups.

\paragraph{Further remarks and questions.}
As commented above, the coboundedness assumption plays an important (though technical) role in the proofs. It would be interesting to know whether it can be avoided. In particular, this raises the following question.
\begin{que}
	Assume $n\geq 3$, and let $\Gamma$ (resp.\ $\Lambda$) be a uniform (resp.\ non-uniform) lattice in SO$(n,1)$. For what values of $p\in [1,\infty]$ are $\Gamma$  and $\Lambda$ $\LL^p$ measure equivalent? \end{que}
Finally, let us recall once more the following result of Sauer: any two amenable finitely generated groups are quasi-isometric if and only if they admit an $\LL^{\infty}$ measure equivalence coupling which is mutually cobounded.  
In general, it is unknown whether being $\LL^{\infty}$ measure equivalent implies being quasi-isometric, even for amenable groups. 
This justifies the following question.

\begin{que}
	Is hyperbolicity invariant under $\LL^{\infty}$ measure equivalence?
\end{que}

\paragraph{Plan of the paper.}

Theorem~\ref{thmIntro:NonhyperbolicQuant} is first proved in \S~\ref{sec:geometricprelim}. 
After recalling the definitions of quantitative measure equivalence in \S~\ref{sec: prelim quant},
we prove our main result, namely Theorem~
\ref{thm:hyperbolicExpME} from which we deduce the theorems announced in the 
introduction.

\section{Geometric preliminaries}\label{sec:geometricprelim}

Let $X$ be a graph, a \textbf{discrete path} in $X$ of length $l\geq 1$ is a map $\alpha\colon \{0,\ldots,l\}\to X$ such that for all $i\in\{0,\ldots,l-1\}$, we have that $\alpha(i)$ and $\alpha(i+1)$ are connected by an edge.  We will also say that $\alpha$ is a path from $\alpha(0)$ to $\alpha(l)$, and we will often identify a path to its range. 

Every connected graph $X$ is viewed as a metric space $(X,d)$ equipped with the \textbf{discrete path metric}, defined by setting $d(x,y)$ as the minimum length of a path from $x$ to $y$. Any discrete path which realizes the discrete path metric between two points is called a \textbf{discrete geodesic}, and it is then an isometric embedding from $\{0,\ldots,d(x,y)\}$ to $(X,d)$. 

Another important metric space that we can get out of a connected graph $X$ is given by the (continuous) \textbf{path metric} which we define as in  \cite[1.15$_+$]{gromovMetricStructuresRiemannian2007}. 
We first identify each edge to the interval $[0,1]$ isometrically, thus obtaining a \emph{length structure} on our graph. The metric associated to this length structure is denoted by $d_l$, and it is by definition the continuous path metric on $X$. It agrees with the discrete path metric on the vertices of $X$, and it is geodesic. Every geodesic between vertices defines a discrete geodesic, and every discrete geodesic can be lifted to a geodesic between vertices.\\

A (geodesic) \textbf{triangle} in a metric space $(X,d)$ with \textbf{vertices} $a_1,a_2,a_3 \in X$ is a set $[a_1,a_2,a_3] \subseteq X$ obtained by taking the union of a choice of geodesics $[a_1,a_2]$, $[a_2,a_3]$, and $[a_3,a_1]$ between its vertices. In the same way, we define a (geodesic) \textbf{$n$-gon} with vertices $a_1, \dots, a_n \in X$, and denote it by $[a_1, \dots, a_n]$. Given an $n$-gon where $n \geq 3$, we will frequently call any of its defining geodesics a \textbf{side}. 

Now, recall that a geodesic space $(X,d)$ is \textbf{$\delta$-hyperbolic} in the sense of Rips if there exists a $\delta \geq 0$ such that for every geodesic triangle $[a_1,a_2,a_3]$ and for every $x \in [a_1,a_2]$, there exists an element $y$ in either $[a_1,a_3]$ or $[a_2,a_3]$ such that $d(x,y)\le \delta$; or equivalently, that the side $[a_1,a_2]$ is contained in the $\delta$-neighborhood of $[a_1,a_3]\cup[a_2,a_3]$.
Moreover, we say that a geodesic space $(X,d)$ is \textbf{hyperbolic} if it is $\delta$-hyperbolic for some $\delta$, and that a finitely generated group $\Gamma$ with generating set $S$ is hyperbolic whenever its Cayley graph 
is hyperbolic when equipped with the continuous path metric.

We shall need the following well-known lemma.

\begin{lemma}\label{lem:distanceGeoPath}
Let $X$ be a  $\delta$-hyperbolic geodesic space, let  $\alpha$ be a path of length $\ell\geq 1$ between two points $x_1$ and $x_2$, and let $y$ belong to a geodesic from $x_1$ to $x_2$. Then 
\[d(y,\alpha)\leq \delta\log_2(\ell)+1.\] 
\end{lemma}
\begin{proof}
Let us prove it by induction on $\lfloor\ell\rfloor$. The case $\lfloor \ell\rfloor=1$ is clear.  So assume $n\geq 2$ and suppose that the lemma is true for all paths of length $<n$. Let $\alpha$ be a path of length $\ell\in[n,n+1[$ from $x_1$ to $x_2$, represented as a gray path in the following figure. For every $x\in \alpha,$ using that the geodesic triangle $[x_1,x,x_2]$ is $\delta$-thin, then either $d(y,[x_1,x])]\leq \delta$ or $d(y,[x,x_2])]\leq \delta$. By connectedness of $\alpha$, there exists $x_3$ such that both conditions are satisfied. By exchanging $x_1$ and $x_2$ if necessary, we can assume that the portion $\alpha_1$ of $\alpha$ from $x_1$ to $x_3$ has length $\ell_1\leq \ell/2$. 

	\begin{center}
	\includegraphics[scale=1]{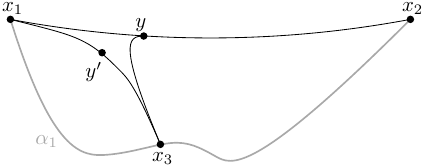}
	\end{center}

Hence there exists a point $y'\in [x_1,x_3]$ such that $d(y,y')\leq \delta$.  
Now applying the induction hypothesis to the path $\alpha_1$ and the point $y'\in[x_1,x_3]$, we obtain 
\[
d(y',\alpha_1)\leq \delta\log_2(\ell_1)\leq \delta\log_2\ell+1 -\delta. 
\]
We deduce by the triangle inequality 
\[
d(y,\alpha)\leq d(y,y')+ d(y',\alpha_1)\leq    \delta\log_2(\ell)+1.
\]
So we are done.
\end{proof}

First we need an alternative definition of hyperbolicity in terms of embedded cycles. In what follows, the \textbf{cycle} $C_n$ of length $n\geq 2$ is the Cayley graph of $\mathbb{Z} / n \mathbb{Z}$ with respect to the generating set containing only the element $1$ mod $n$, which we view as a \emph{discrete} metric space denoted  by $(C_n, d_{C_n})$.

\begin{proposition}\label{prop:HypNotEmbed}
	Let $(X,d)$ be a $\delta$-hyperbolic geodesic space, let $n$ be a positive integer. Then for every $a\geq 0$ and every $b\geq 1$,  if there is a  map $\varphi\colon C_{2n}\to X$ such that  for every $x,y\in C_{2n},$
	\[ad_{C_{2n}}(x,y)  \leq d(\varphi(x),\varphi(y)) \leq bd_{C_{2n}}(x,y)\]
	  then we have 
\begin{equation*}\label{eq:hypcycle}
a\leq \frac{4\delta\log_2(bn)+4+2b}{n}
\end{equation*}
\end{proposition}
Before proving the above proposition, let us note the following straightforward corollary, 
using the estimate $\frac 1{\log 2}< \frac32$.
\begin{corollary}\label{cor:HypNotEmbed}
	Let $a\geq 0$, $b\geq 1$. There is an integer $n_0=n_0(b)\geq 2$ such that the following holds. For all
	$n\geq n_0$ and all $\delta\geq 1$, if $X$ is a  $\delta$-hyperbolic geodesic space and if there is a  map $\varphi\colon C_n\to X$ such that for all $x,y\in C_n$
	\[ad_{C_n}(x,y) \leq d(\varphi(x),\varphi(y)) \leq bd_{C_n}(x,y)\]
	then we have 
\begin{equation}\label{eq:hypcycle'bis}
a< 6\delta\cdot  \frac{\log n}{n}.
\end{equation}
\end{corollary}

In order to prove the proposition, we need the following additional notion. Given a discrete path $\beta$ in a graph $Y$ and a map $\varphi\colon Y\to(X,d)$ where $(X,d)$ is geodesic, we say that a continuous path $\alpha$ in $X$ is a $\varphi$\textbf{-direct image} of $\beta$ if it is obtained by concatenating geodesics between $\varphi(\beta(i))$ and $\varphi(\beta(i+1))$ where $i$ ranges from $0$ to $\ell(\beta)-1$.

\begin{proof}[Proof of \cref{prop:HypNotEmbed}]
	Let $a_1,a_2\in C_{2n}$ be such that $d_{C_{2n}}(a_1,a_2)=n$ and define $x_1=\varphi(a_1)$, $x_2=\varphi(a_2)$. Consider a geodesic $[x_1,x_2]$ from $x_1$ to $x_2$. In $C_{2n}$ there are two discrete geodesic paths from $a_1$ to $a_2$, both with length $n$. Denote by  $\alpha$ and $\alpha'$ some respective $\varphi$-direct images of those paths in $X$, which by assumption have length at most $bn$. Let $y\in [x_1,x_2]$, we deduce from \cref{lem:distanceGeoPath} that 
	\[d(y,\alpha)\leq \delta\log_2(\ell(\alpha))+1\leq \delta\log_2\left(bn\right)+1,\]
	 where $\ell(\alpha)$ is the length of $\alpha$, and by the same argument $d(y,\alpha')\le \delta\log_2(bn)+1$.
	If we then pick $z_y$ and $z'_y$ points in $C_{2n}$ such that their $\varphi$-images are in $\alpha$ and $\alpha'$ respectively and minimize the distance to $y$, we have
	\begin{equation}\label{ineq: dist phizy}
	\max\left(d(y,\varphi(z_y)),d(y,\varphi(z'_y)\right)\leq \delta\log_2(bn)+1+\frac b2.
	\end{equation}
	
	Note that for any $y\in[x_1,x_2]$, any geodesic from $z_y$ to $z'_y$ must pass through $x_1$ or through $x_2$. Moreover there are some $y\in[x_1,x_2]$ for which the first case occurs, and some for which the second case occurs. For all $\eps>0$, we may thus find $y_1,y_2\in[x_1,x_2]$ such that $d(y_1,y_2)\le \eps$, the geodesic from $z_{y_1}$ to $z'_{y_1}$ passes through $x_1$ and the geodesic from $z_{y_2}$ to $z'_{y_2}$ passes through $x_2$. 
	
	\begin{center}
		\includegraphics[scale=1]{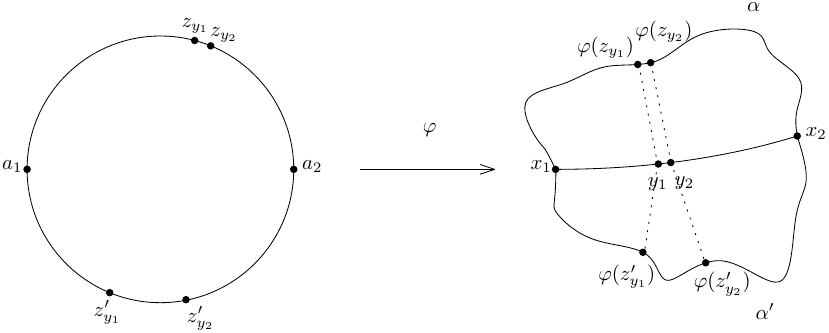}
	\end{center}
	Then we have that \[d_{C_n}(z_{y_1},z_{y_2})+d_{C_n}(z_{y_2},z'_{y_2})+d_{C_n}(z'_{y_2},z'_{y_1})+d_{C_n}(z'_{y_1},z_{y_1})= 2n.\]
	Hence one of these four distances is at least $\frac{n}{2}$. On the other hand, combining \eqref{ineq: dist phizy}  and the fact that $d(y_1,y_2)\leq\eps$, we obtain the following inequality
	\[d(\varphi(z_{y_1}),\varphi(z_{y_2}))+d(\varphi(z_{y_2}),\varphi(z'_{y_2}))+d(\varphi(z'_{y_2}),\varphi(z'_{y_1}))+d(\varphi(z'_{y_1}),\varphi(z_{y_1}))\leq 2(\delta\log_2\left(bn\right)+1+\frac b 2)+\eps.\] Using our assumption on $\varphi$, we thus have the following inequality: for all $\eps>0$,
	\begin{align*}
		\frac{an}{2}&\leq 2\delta\log_2\left(bn\right)+2+b +\eps.
	\end{align*}
	So the proposition follows.
\end{proof}
 
\section{Preliminaries on quantitative measure equivalence}\label{sec: prelim quant}

We now make the statements from the introduction more precise
by 
enriching a bit the terminology from
\cite[Sec.~2]{delabieQuantitativeMeasureEquivalence2022}
and introducing in details quantitative measure equivalence.
We denote systematically smooth actions
by $\actom$ (recall that by definitions, smooth actions
are those which admit a Borel fundamental domain, i.e.~a Borel
subset intersecting each orbit exactly once).

\subsection{Relations between fundamental domains}\label{sec: rel between fund dom}

In this section, we fix a smooth measure-preserving action 
$\Gamma\act (\Omega,\mu)$.

Let $X_1$, $X_2$ be two fundamental domains, we denote by 
$\pi_{X_1,X_2}:X_1\to X_2$ the map which takes every $x\in X_1$
to the unique $x'\in X_2\cap \Gamma\actom x$.
The map $\pi_{X_1,X_2}$ belongs to the pseudo full group of the action, 
in particular it is measure-preserving, and its inverse is
$\pi_{X_2,X_1}$. 

Say that two fundamental domains $X_1$ and $X_2$ are 
$\LL^\infty$-equivalent if there is a finite subset $F\Subset\Gamma$
such that for all $x\in X_1$, there is $\gamma\in F$ such that 
such that $\pi_{X_1,X_2}(x)=\gamma\actom x$. 
Observe that $\LL^\infty$-equivalence is an equivalence relation, 
and that if some measure-preserving $T\in \Aut(\Omega,\mu)$ 
commutes with the $\Gamma$-action, then $X_1$ is $\LL^\infty$-equivalent
to $X_2$ iff $T(X_1)$ is $\LL^\infty$-equivalent to $T(X_2)$. 

Now let $\varphi: \R^+\to\R^+$ be a non-decreasing function
and assume that $\Gamma$ is generated by a finite set $S_\Gamma$, allowing 
us to endow $\Omega$ with the Schreier metric $d_{S_\Gamma}$ 
whose definition is recalled in~\cite[Def.~2.14]{delabieQuantitativeMeasureEquivalence2022}.
For the purpose of this paper, 
we introduce some further terminology and say that two 
fundamental domains $X_1,X_2$ are $\varphi$-\textbf{similar}
if 
\[
\int_{X_1}\varphi(d_{S_\Gamma}(x,\pi_{X_1,X_2}(x)))d\mu(x)<+\infty.
\]
Observe that $\LL^\infty$-equivalence can be recast as the fact that the map
$x\in X_1\mapsto d_{S_\Gamma}(x,\pi_{X_1,X_2}(x))$ takes only finitely 
many values, in particular it implies $\varphi$-similarity. 
However, as the name suggests, $\varphi$-similarity is not an equivalence relation 
in general, for instance when $\varphi(t)=e^t$ transitivity might fail
(see also \cite[Rem.~2.16]{delabieQuantitativeMeasureEquivalence2022}).
Even worse, it is a priori dependent on the choice of the finite 
generating set 
$S_\Gamma$ we made. 

In order to correct this, we introduce as in 
\cite{delabieQuantitativeMeasureEquivalence2022} a coarser relation
that we called $\varphi$-equivalence: two fundamental domains $X_1$ and $X_2$
are $\varphi$\textbf{-equivalent} if there is some $\delta>0$ such that they 
are $\varphi_\delta$-similar, where $\varphi_\delta(t)=\varphi(\delta t)$.
We checked in \cite[Cor.~2.19]{delabieQuantitativeMeasureEquivalence2022}
that $\varphi$-equivalence is an equivalence relation. 

\begin{remark}\label{rmk: equivalence vs similarity}
    It is important to note that \(\varphi\)-equivalence is equivalent to \(\varphi\)-similarity
when \(\varphi\) satisfies that for all \(c>0\), there is \(C>0\) such that for all \(t\geq 0\) we have \(\varphi(ct)\leq C\varphi(t)\). This is notably the case
when \(\varphi(t)=t^p\) for some \(p>0\), and we will make use of this fact
without explicit mention. 
\end{remark}

\subsection{\texorpdfstring{$(\varphi,\psi)$}{phi,psi}-integrable measure equivalence}\label{sec:quantme}

In this section we fix two finitely generated groups $\Gamma=\la S_\Gamma\ra$ 
and $\Lambda=\la S_\Lambda\ra$. 

As recalled in the introduction, a \textbf{measure equivalence coupling} from $\Gamma$ to $\Lambda$
is a measured 
space $(\Omega,\mu)$ endowed with commuting smooth free measure-preserving actions
of $\Gamma$ and $\Lambda$ and Borel 
fundamental domains $X_\Gamma$ for the $\Gamma$-action, $X_\Lambda$ for the $\Lambda$-action,
which both have finite measure. 
We write such couplings as $(\Omega,\mu,X_\Gamma,X_\Lambda)$.

\begin{definition}\label{def: phi psi meq cocycle free}
	Let $\varphi,\psi:\R^+\to\R^+$ be two non-decreasing functions.
	We say that a measure-equivalence coupling $(\Omega,\mu,X_\Gamma,X_\Lambda)$ from $\Gamma$ to $\Lambda$ is
	$(\varphi,\psi)$\textbf{-integrable} if for every $\gamma\in S_\Gamma$
	the $\Lambda$-fundamental domain $\gamma\actom X_{\Lambda}$ is $\varphi$-equivalent to $X_\Lambda$
	, and for every 
	$\lambda\in S_\Lambda$ the $\Gamma$-fundamental domain $\lambda\actom X_\Gamma$ is $\psi$-equivalent to $X_\Gamma$.
\end{definition}

As explained in \cite[Prop.~2.22]{delabieQuantitativeMeasureEquivalence2022}, since $\varphi$-equivalence and $\psi$-equivalence are equivalence relations,
the definition is unchanged if we quantify over all $\gamma\in\Gamma$ or all $\lambda\in\Lambda$ rather than over the finite generating sets $S_\Gamma$ and 
$S_\Lambda$.

We now connect this to the cocycle version of these definitions, already given in the introduction:  given a measure equivalence coupling $(\Omega,\mu,X_\Gamma,X_\Lambda)$, we
have 
the associated cocycles 
\[
\alpha:X_\Lambda\times \Gamma\to\Lambda\text{ and } \beta: X_\Gamma\times \Lambda\to\Gamma
\]
uniquely defined by the following statements: 
for all $x\in X_\Lambda$ and $\gamma\in\Gamma$, 
$\alpha(\gamma,x)\actom (\gamma\actom  x)\in X_\Lambda$, 
and similary for all $x\in X_\Gamma$ and $\lambda\in\Lambda$,
$\beta(\lambda,x)\actom(\lambda\actom x)\in X_\Gamma$. 

When we endow $\Gamma$ and $\Lambda$ with the natural norms $\abs\cdot_{S_\Gamma}$
and $\abs\cdot_{S_\Lambda}$ associated to 
their respective generating sets $S_\Gamma$ and $S_\Lambda$, we can now state 
that a measure equivalence coupling $(\Omega,\mu,X_\Gamma,X_\Lambda)$ from $\Gamma$ to $\Lambda$ is $(\varphi,\psi)$\textbf{-integrable} iff the
associated cocycles satisfy: there is $\delta>0$ such that for all $s\in S_\Gamma$ and $t\in S_\Lambda$ we have
\[
\int_{X_\Lambda} \varphi(\delta \abs{\alpha(s,x)}_{S_\Lambda})d\mu(x)<+\infty\quad \text{ and }\quad \int_{X_\Lambda} \psi(\delta\abs{\beta(t,x)}_{S_\Gamma})d\mu(x)<+\infty.
\]

\begin{remark}As in the end of the previous section, note that if \(\varphi\) satisfies that for all \(c>0\), there is \(C>0\) such that for all \(t\geq 0\) we have \(\varphi(ct)\leq C\varphi(t)\), then we can get rid of the factor 
\(\delta\) in the first inequality, and the same applies to the second inequality
mutatis mutandis. This applies in particular for \(\varphi(t)=t^p\), and thus 
for \(\LL^p\) conditions, but not for exponential integrability.
\end{remark}

Let us denote by $\cdot$ the natural $\Gamma$ (resp.~$\Lambda$) action on $X_\Lambda$ (resp.~$X_\Gamma$) given by : for all $\gamma\in\Gamma$ and $x\in X_\Lambda$, $\gamma\cdot x$ is the unique element of $\Lambda\actom(\gamma\actom x)\cap X_\Lambda$ (and symmetrically for all $\lambda\in\Lambda$ and all $x\in X_\Gamma$, $\lambda\cdot x$ is the unique element of $\Gamma\actom(\lambda\actom x)\cap X_\Gamma$). Observe that for all $x\in X_\Lambda$ and $\gamma\in\Gamma$, the cocycle \(\alpha(\gamma,x)\) is uniquely defined by the equation
$\alpha(\gamma,x)\actom (\gamma\actom  x)=\gamma\cdot x$, and a similar statement holds for \(\beta\).

The equivalence between the above definition and Definition~\ref{def: phi psi meq cocycle free} is then clear once one notes that 
$\alpha$ is connected to the cocycle of $\pi_{X_\Lambda,\gamma\actom X_\Lambda}$: it is uniquely defined by the equation
\[\pi_{X_\Lambda,\gamma\actom X_\Lambda}(x)=\alpha(\gamma,\gamma\inv\cdot x) \actom x.\]
Indeed it follows that $d_{S_\Gamma}(x,\pi_{X_\Lambda,\gamma\actom X_\Lambda}(x))=\abs{\alpha(\gamma,\gamma\inv\cdot x)}_{S_\Gamma}$, so since the action of $\gamma\inv$ on $X_\Lambda$ is measure-preserving, the $\varphi$-integrability of $x\mapsto d_{S_\Gamma}(x,\pi_{X_\Lambda,\gamma\actom X_\Lambda}(x))$ is equivalent to that of $x\mapsto \abs{\alpha(\gamma,x)}_{S_\Gamma}$ (and the symmetric phenomenon holds for $\beta$).

Finally, let us recall the cocycle relations satisfied by \(\alpha\) (and \(\beta\) mutatis mutandis): for all \(x\in X_\Lambda\) and all \(\gamma,\gamma'\in\Gamma\) we have
\[
\alpha(\gamma'\gamma,x)=\alpha(\gamma',\gamma\cdot x)\alpha(\gamma,x).
\]

\section{Rigidity of hyperbolicity}\label{sec:rigidityhyper}

We now prove rigidity results, saying that hyperbolicity is preserved under cobounded measure equivalence couplings satisfying certain integrability conditions. To simplify the exposition, we choose to state two results, but the second one should be seen as a degenerate version of the first one (where one of the conditions becomes an $L^{\infty}$-condition). Anyway the two results share basically the same proof, which consists in confronting \cref{prop:HypNotEmbed} with the fact that 
non hyperbolic graphs contain arbitrarily large $18$-bi-Lipschitz embedded cycles \cite[Prop.~5.1]{humePoorlyConnectedGroups2020}. 
Since our two results are quite technical, we shall first deduce two more appealing corollaries (namely \cref{thmintroHyper} and \cref{thmintroHyper2} from the introduction). 

Let us start with some notation: given a group $\Gamma$ equipped with a symmetric finite generating subset $S_\Gamma$, we denote the growth function $\vol_{S_\Gamma}(n)=|S_\Gamma^n|$ and define the associated 
\textbf{volume entropy}  as $\ent(S_\Gamma)=\limsup_{r\to \infty} \frac{\log (\vol_{S_\Gamma}(r))}{r}$.
\begin{theorem} \label{thm:hyperbolicExpME}
	Let $\Gamma$ and $\Lambda$ be two finitely generated groups such that $\Gamma$ is $\delta$-hyperbolic.
	We let $L\geq 1$ and $\varphi$, $\psi$ and $r$ be non-decreasing unbounded functions. Assume that the following conditions are satisfied:
	\begin{equation}\label{eq:cond1}
	\lim_{n\to \infty}\frac{n^2r(n)\vol_{S_\Gamma}(r(n))}{\varphi(n/r(n))}= 0;
	\end{equation}
	and\footnote{The constant 18 in (\ref{eq:cond2}) is the bi-Lipschitz constant from \cite[Prop.~5.1]{humePoorlyConnectedGroups2020}, while the right-hand term comes from the use of \cref{prop:HypNotEmbed}.} for all large enough $n$,
	\begin{equation}\label{eq:cond2}
	\frac{r(n)}{18}\ge 4(\delta+1)\log_2 n+3\psi^{-1}(3Ln).
	\end{equation}	 
	Assume that $(\Omega,X_\Gamma,X_\Lambda,\mu)$ 
    is a cobounded measure equivalence coupling from 
	$\Gamma$ to $\Lambda$, normalized so that $\mu(X_\Gamma)=1$, and such that 
	the associated cocycles $\alpha\colon \Gamma\times X_\Lambda\to \Lambda$ 
	and $\beta\colon \Lambda\times X_\Gamma \to \Gamma$ satisfy the following 
	properties. 
	\begin{itemize}
		\item[(i)] for all $s\in S_\Gamma$,	\[\int_{X_\Lambda} \varphi(|\alpha(s,x)|_{S_\Lambda}) d\mu(x)< \infty;\]
		\item[(ii)]  for all $t\in S_\Lambda$
		\[\int_{X_\Gamma} \psi(|\beta(t,x)|_{S_\Gamma})d\mu(x)\leq L.\]
	\end{itemize}
	Then $\Lambda$ is hyperbolic.
	\end{theorem}

Note that our assumption on the coupling is a slight strengthening of $(\varphi,\psi)$-integrability where 
the multiplicative constant \(\delta>0\) is taken to be equal to \(1\).
Assuming that $\beta$ is bounded, we have the following variant.
\begin{theorem} \label{thm:hyperbolicExpMELinfty}
	Let $\Gamma$ and $\Lambda$ be two finitely generated groups such that $\Gamma$ is $\delta$-hyperbolic.
	We let  $\varphi$ and $r$ be increasing unbounded functions. Assume that 
(\ref{eq:cond1}) holds, and that for all large enough $n$,
	\begin{equation}\label{eq:cond2'}
	\frac{r(n)}{18}\geq 6 (\delta+1) \log n.
	\end{equation}	
	Assume that $(\Omega,\mu)$ is a cobounded measure equivalence coupling from
	$\Gamma$ to $\Lambda$, normalized so that $\mu(X_\Gamma)=1$, and such that 
	the associated cocycles $\alpha\colon \Gamma\times X_\Lambda\to \Lambda$ 
	and $\beta\colon \Lambda\times X_\Gamma \to \Gamma$ satisfy the following 
	properties. 
	\begin{itemize}
		\item[(i)]\label{item:cond1'} for all $s\in S_\Gamma$,	\[\int_{X_\Lambda} \varphi(|\alpha(s,x)|_{S_\Lambda}) d\mu(x)< \infty;\]
		\item[(ii)]\label{item:cond2'} for all $t\in S_\Lambda$ 
		\[|\beta(t,\cdot)|_{S_\Gamma}\in \LL^{\infty}(X_\Gamma).\]
	\end{itemize}
	Then $\Lambda$ is hyperbolic.	
	\end{theorem}

Theorem~\ref{thmintroHyper} is an immediate consequence of the following corollary of Theorem~\ref{thm:hyperbolicExpMELinfty}.
\begin{corollary}\label{cor:Lp-infty}
	Let $\Gamma$ be a finitely generated $\delta$-hyperbolic group and let $p>108\delta \ent(S_\Gamma)+2$. Assume that there exists a cobounded $(\LL^p,\LL^{\infty})$-integrable measure equivalence coupling from  $\Gamma$ to a finitely generated group $\Lambda$. Then $\Lambda$ is hyperbolic.
		\end{corollary}
\begin{proof}
We apply \cref{thm:hyperbolicExpMELinfty} with $p=108\delta (\ent(S_\Gamma)+\eps)+2$ for some $\eps>0$. We let $r(n)=108\delta \log n$. By definition of $\ent(S_\Gamma)$, we  have $\vol_{S_\Gamma}(r(n))=o(e^{(\ent(S_\Gamma)+\eps/2)r(n)})$. Hence we deduce that 
\[\vol_{S_\Gamma}(r(n))=o(n^{108\delta(\ent(S_\Gamma)+\eps/2)}),\]
which combined with the fact that $\varphi(t)=t^{108\delta (\ent(S_\Gamma)+\eps)+2}$ implies that (\ref{eq:cond1}) is satisfied. Thus, $\Lambda$ is hyperbolic.
	\end{proof}
\begin{corollary}\label{cor: phi coupling hyp}
	Let $\Gamma$ be a finitely generated hyperbolic group. For every $p>q>0$ such that if there is a cobounded $(\varphi,\psi)$-integrable measure equivalence coupling  from $\Gamma$ to a finitely generated group $\Lambda$, where 
	$\varphi(t)=\exp(t^p)$ and $\psi(t)=t^{1+1/q}$, then $\Lambda$ is also hyperbolic.  
\end{corollary}
\begin{proof}
	We consider our cobounded $(\varphi,\psi)$-integrable measure equivalence coupling $(\Omega,X_\Gamma,X_\Lambda,\mu)$ and normalize $\mu$ so that $\mu(X_\Gamma)=1$. We pick $\eta$ strictly between $q$ and $p$, we let $r(n)=n^{\frac{\eta}{1+\eta}}$, and we define  
	\[L=\max_{s\in S_\Lambda}\int_{X_\Gamma}\psi(|\beta(s,x)|_{S_\Gamma})d\mu(x).\]
	Note that $\psi^{-1}(t)=t^{\frac{q}{1+q}}$. So (\ref{eq:cond2}) follows from the fact that $\eta>q$.
	Finally, take $\eps>0$ such that 
	\[\max_{s\in S_\Gamma}\int_{X_\Lambda}\varphi_\eps(|\alpha(s,x)|_{S_\Lambda})d\mu(x)<\infty,\]
	where $\varphi_\eps(t) = \varphi(\eps t)$.
Note that $n/r(n)=n^{\frac{1}{1+\eta}}$. Hence $\varphi_{\eps}(n/r(n))= \exp(\eps^p n^{\frac{p}{1+\eta}})$, while 
	\[\vol_{S_\Gamma}(r(n))\leq |S_\Gamma|^{n^{\frac{\eta}{1+\eta}}}\]
Hence since $\eta<p$, we have 
\[\frac{\vol_{S_\Gamma}(r(n))}{\varphi_{\eps}(n/r(n))}=O(n^{-k})\]
for any $k>0$. So (\ref{eq:cond1}) is satisfied and we  conclude by \cref{thm:hyperbolicExpME} that $\Lambda$ is hyperbolic.
\end{proof}

\begin{proof}[Proof of Theorems~\ref{thm:hyperbolicExpME} and~\ref{thm:hyperbolicExpMELinfty}]
	Recall that 
	we use the notation $\actom$ for smooth actions and
	\(\cdot\) for the induced actions on the  respective fundamental domains. 
	Let us start by strengthening the coboundedness condition. 
	
	\begin{clai}
		 Replacing Condition (\ref{eq:cond2}) and Condition (\ref{eq:cond2'}) respectively by the slightly weaker conditions 
		 \begin{equation}\label{eq:cond2"}
		 \frac{r(n)}{18}\geq 4\delta\log_2 n+3\psi^{-1}(3Ln)-3,
		 \end{equation} 
		 and
 \begin{equation}\label{eq:cond2'''} 
	\frac{r(n)}{18}\geq 6 \delta\log n,
	\end{equation}	
 it is enough to prove  \cref{thm:hyperbolicExpME} and \cref{thm:hyperbolicExpMELinfty} under the assumption that $X_\Gamma\subseteq X_\Lambda$.  
	\end{clai}
	\begin{cproof}
		Assuming we have a coupling satisfying the conditions of Theorem~\ref{thm:hyperbolicExpME} 
		(resp.\ \cref{thm:hyperbolicExpMELinfty}), 
		we build a new coupling satisfying Condition~\eqref{eq:cond2"} (resp.\ Condition~\eqref{eq:cond2'''}) with 
		$X_\Gamma\subseteq X_\Lambda$.
		
		Since our initial coupling is cobounded, there exists a finite subset $F$ of 
		$\Lambda$ such that $X_\Gamma\subseteq F\actom X_\Lambda$. Consider the 
		new coupling space $\tilde \Omega\colon=\Omega\times F$, let $K$ be a finite 
		group which acts simply transitively on $F$, and let $\tilde 
		\Gamma=\Gamma\times K$ act on $\tilde \Omega$ by 
		$(\gamma,k)\actom(\omega,f)=(\gamma\actom \omega,kf)$. This 
		action is smooth, and we take as a fundamental domain the set 
		\[
		\tilde X_{\tilde\Gamma}\colon=\bigsqcup_{f\in F} (X_\Gamma\cap f\actom X_\Lambda)\times \{f\}
		\]	
		The $\Lambda$-action on $\tilde \Omega$ is the action on the first coordinate; a fundamental domain is provided by 
		\[
		\tilde X_\Lambda=\bigsqcup_{f\in F}(f\actom X_\Lambda)\times \{f\}.
		\] 
		Viewing both $\Gamma$ and $K$ as subgroups of $\tilde\Gamma$, the latter has $S_{\tilde\Gamma}=S_\Gamma\cup K$ as a finite generating set. In fact, with this generating set $\tilde \Gamma$ is $\tilde{\delta}$-hyperbolic, with $\tilde{\delta}=\delta+1$. It follows that Condition \eqref{eq:cond2'''} holds in this new setup.
		 Also observe that the volume growth of $\tilde\Gamma$ is at most $|K|$ times that of $\Gamma$, so Condition \eqref{eq:cond1}  is preserved.

		 In what follows, we implicitely use the fact that 
		 our quantitative conditions (i) and (ii) can be recast using the notion of 
		 $\varphi$-similarity between fundamental domains of a smooth action
		 as explained in Section \ref{sec: prelim quant}. 
		 We also use the straightforward fact that
		 $\LL^\infty$-equivalence
		refines $\varphi$-similarity.
		
		We can now show that condition (i) is still met by the new generating set 
		$S_{\tilde\Gamma}=S_\Gamma\cup K$. 
		Indeed $\tilde X_\Lambda$ is $\LL^\infty$-equivalent to the fundamental 
		domain $X_\Lambda\times F$, and for all $\gamma\in\Gamma$, 
		$x\in X_\Lambda$ and $f\in F$ we have
		$d_{S_\Lambda}(\gamma\actom(x,f),\gamma\actx(x,f))=d_{S_\Lambda}(\gamma\actom x,\gamma\actx x)$, 
		so for all $\tilde\gamma\in S_{\tilde \Gamma}$ we have that $\tilde\gamma\actom X_\Lambda\times F$	is $\varphi$-similar to $X_\Lambda\times F$, so  $\tilde\gamma\actom\tilde X_\Lambda$ is $\varphi$-similar to $\tilde X_\Lambda$.

		For condition (ii), we have, by construction, for all $x\in X_\Gamma$ and all $f,f'\in F$ the inequality $d_{S_{\tilde \Gamma}}((x,f),(x,f'))\leq 1$, so for every $\lambda\in \Lambda$ we have
		\[
		d_{S_{\tilde\Gamma}}(\lambda\cdot (x,f),\lambda\actom(x,f))\leq 1+d_{S_\Gamma}(\lambda\cdot x,\lambda\actom x) = 1 + \abs{\beta(\lambda,x)}_{S_\Gamma},
		\]
		hence the new coupling satisfies the same conditions replacing $\psi(t)$ by $\tilde{\psi}(t)=\psi(\max\{t-1,0\})$ in \cref{thm:hyperbolicExpME}. Note that for $n$ large enough, $\psi^{-1}(3Ln)\geq 1$.  Since $\tilde{\psi}(t) = \psi(t-1)$ for all $t\ge 1$, we deduce that for large enough $n$,
\[\frac{r(n)}{18} \geq 4\tilde{\delta}\log_2 n+3\tilde{\psi}^{-1}(3Ln)-3.\]				
So the claim is proved.
	\end{cproof}
From now on, we thus assume that $X_\Gamma\subseteq X_\Lambda$ and  we normalize the measure so that  $\mu(X_\Gamma)=1$.
	Suppose by contradiction that $\Lambda$ is not hyperbolic.
	Theorem~\ref{thmIntro:Nonhyperbolic} provides us with a cycle $C_n$ of arbitrary 
	large length $n$,  and a map $C_n\to \Lambda$ 
	which is $1$-Lipschitz and contracts distances at most by a factor $18$. 
	In what follows we consider $C_n$ as a subset of $\Lambda$. Let $K$ be such that $\int_{X_\Lambda} \varphi(|\alpha(s,x)|_{S_\Lambda}) d\mu(x)\leq K$ for all $s\in S_\Gamma$. 

	For every $x\in X_\Gamma$ we denote by $b_x\colon\Lambda\to \Gamma$ the map defined by $b_x(\lambda)=\beta(\lambda^{-1},x)^{-1}$ for every $\lambda\in\Lambda$.
	We will use throughout the following straightforward consequence of the cocycle relation: for all $u,v\in \Lambda$, we have
	\begin{equation}\label{eq:cocycleb}
	b_x(u)^{-1}b_x(v)=\beta(v^{-1}u,u^{-1}\cdot x)^{-1}.
	\end{equation}
	We endow \(\Gamma\) and \(\Lambda\) with their usual left-invariant Cayley metrics,
	denoted by \(d_{S_\Gamma}\) and \(d_{S_\Lambda}\) respectively (so the map
	\(\gamma\mapsto \gamma\inv\actom x\) is an isometry if we endow \(\Gamma\actom x\) with
	the Schreier metric that we previously used and denoted by \(d_{S_\Gamma}\) as well).
	
	\paragraph{Upper estimates for the restriction of $b_x$ to $C_n$.}
	In the context of \cref{thm:hyperbolicExpMELinfty}, the boundedness
    of each $\beta(t,\cdot)$ for $t\in S_\Lambda$ and the cocycle identity yield that $b_x$ is a.e.\ $L$-Lipschitz for some constant $L$.

	On the other hand, under the assumption of \cref{thm:hyperbolicExpME}, we claim that with probability at least $2/3$, the restriction of $b_x$ to $C_n$ is  $\psi^{-1}\left(3Ln\right)$-Lipschitz.
	For every $u$ and $v$ adjacent in $C_n$ there exists an $s_{u,v}\in S_\Lambda$ such that $u=vs_{u,v}$. By (\ref{eq:cocycleb}), we have: \[d_{S_\Gamma}(b_x(v),b_x(u)) = |b_x(u)^{-1}b_x(v)|_{S_\Gamma}= |\beta(s_{u,v},u^{-1}\cdot x)|_{S_\Gamma}.\] Next consider for any $u\in C_n$ the set of all $x\in X_\Gamma$ such that $\psi(|\beta(s_{u,v},u^{-1}\cdot x)|_{S_\Gamma})\ge 3Ln$. By Markov's inequality, these sets have measure at most $\frac{1}{3n}$ and therefore the set of all $x\in X_\Gamma$ such that $b_x$ is $\psi^{-1}\left(3Ln\right)$-Lipschitz in restriction to $C_n$ has measure at least $1-n\cdot \frac{1}{3n}=\frac{2}{3}$. So our claim follows.

	\paragraph{Lower estimates for the restriction of $b_x$ to $C_n$.}
	Providing lower estimates on the quasi-isometric embedding constants is more involved as this requires to apply the cocycle $\alpha$ to  $C_n$.  We shall use the inverse relation between $\alpha$ and $\beta$ and the inclusion $X_\Gamma\subseteq X_\Lambda$: for all $x\in X_\Gamma$ and $\lambda\in \Lambda$
	\begin{equation}\label{eq:inverse}\alpha(\beta(\lambda,x),x) = \lambda \end{equation}
	We claim that we have the following key inequality.
	\begin{clai}\label{claim:hyperbolic}

		For every $R>0$, and $u$ and $v$ in $\Lambda$ , we have \[\mu\left(\left\{x\in X_\Gamma \colon d_{S_\Gamma}\big(b_x(v),b_x(u)\big)\le R\right\}\right)\le KR\frac{ \vol_{S_\Gamma}(R)}{\varphi\left(\frac{d_{S_\Lambda}(u,v)}{R}\right)},\]
where $K=\max_{s\in S_\Gamma}\int_{X_\Lambda} \varphi(|\alpha(s,x)|_{S_\Lambda}) d\mu(x).$
			\end{clai}
	\begin{cproof}
        Let us fix \(u,v\in\Lambda\).
		For any $\gamma\in \Gamma$, we define the set \[A_\gamma=\{x\in X_\Gamma\colon b_x(u)^{-1}b_x(v)=\gamma^{-1}\}.\] 
		By (\ref{eq:inverse}) and (\ref{eq:cocycleb}), we have that for every $x\in A_\gamma$, \[\alpha\left(\gamma, u^{-1}\cdot x)\right)=\alpha\left(\beta(v^{-1}u,u^{-1}\cdot x),u^{-1}\cdot x\right)=v^{-1}u,\]
		from which we deduce that $\left|\alpha\left(\gamma, u^{-1}\cdot x)\right)\right|_{S_\Lambda} = d_{S_\Lambda}(u,v).$
		
		Let us start giving an upper bound of $\mu(A_\gamma)$ as a function of $|\gamma|$.
		Write $\gamma=s_1\ldots s_{|\gamma|_{S_\Gamma}}$ with $s_i\in S_\Gamma$. By the cocycle relation and the triangular inequality, there exists an $i$ such that the set $$\left\{x\in A_\gamma\colon \left|\alpha\left(s_i,s_{i+1}\ldots s_{|\gamma|_{S_\Gamma}}\cdot( u^{-1}\cdot x)\right)\right|_{S_\Lambda}\ge \frac{ d_{S_\Lambda}(u,v)}{|\gamma|_{S_\Gamma}}\right\}$$
		has measure at least $\frac{\mu(A_\gamma)}{|\gamma|_{S_\Gamma}}$.
		Letting $s=s_i$, we have that \[\mu\left(\left\{y\in X_\Lambda\colon |\alpha(s,y)|_{S_\Lambda}\ge \frac{ d_{S_\Lambda}(u,v)}{|\gamma|_{S_\Gamma}}\right\}\right)
		\ge \frac{\mu(A_\gamma)}{|\gamma|_{S_\Gamma}},\]
		from which we deduce the following upper bound on the measure of $A_\gamma$:
\[\mu(A_\gamma)\le |\gamma|_{S_\Gamma}\,\mu\left(\left\{y\in X_\Lambda\colon |\alpha(s,y)|_{S_\Lambda} \ge  \frac{ d_{S_\Lambda}(u,v)}{|\gamma|_{S_\Gamma}}\right\}\right)\]

By Markov's inequality and the definition of $K=\max_{s\in S_\Gamma}\int_{X_\Lambda} \varphi(|\alpha(s,x)|_{S_\Lambda}) d\mu(x)$, we deduce 
\[\mu(A_\gamma) \le \frac{ K|\gamma|_{S_\Gamma}}{\varphi\left(\frac{d_{S_\Lambda}(u,v)}{|\gamma|_{S_\Gamma}}\right)}.\]                                                                                                                                     
Using that $\varphi$ is non-decreasing, we get that for all $R>0$,
		\begin{align*}
		\mu\left(\{x\in X_\Gamma \colon |b_x(v)^{-1}b_x(u)|_{S_\Gamma}\le R\}\right) & = \sum_{\gamma\in B_\Gamma(e_\Gamma,R)} \mu(A_\gamma)\\
		& \le \sum_{\gamma\in B_\Gamma(e_\Gamma,R)} \frac{ K|\gamma|_{S_\Gamma}}{\varphi\left(\frac{d_{S_\Lambda}(u,v)}{R}\right)}\\
		& \le KR\frac{ \vol_{S_\Gamma}(R)}{\varphi\left(\frac{d_{S_\Lambda}(u,v)}{R}\right)}.
		\end{align*}
		So the claim is proved. \end{cproof}
	Applying \cref{claim:hyperbolic} with $R= \frac{r(n)}{n}d_{S_\Lambda}(u,v)$, $u,v\in C_n$, observing that $d_{S_\Lambda}(u,v)\leq n$, we obtain 
	\begin{align*}
	\mu\left(\left\{x\in X_\Gamma \colon d_{S_\Gamma}(b_x(v),b_x(u))\le \frac{r(n)}{n}d_{S_\Lambda}(u,v)\right\}\right) & \le \frac{Kr(n)\vol_{S_\Gamma}(r(n))}{\varphi(n/r(n))}.
	\end{align*}
	 As there are at most $n^2$ pairs $(u,v)$ in $C_n$, we deduce that 
	\begin{align*}
	\mu\left(\left\{x\in X_\Gamma \colon \exists u,v\in C_n\colon d_{S_\Gamma}(b_x(v),b_x(u))\le \frac{r(n)}{n}d_{S_\Lambda}(u,v)\right\}\right)& \leq \frac{Kn^2r(n)\vol_{S_\Gamma}(r(n))}{\varphi(n/r(n))}.
	\end{align*} 
By (\ref{eq:cond1}), there exists $n_0$ such that for $n\geq n_0$, there exists a subset $B$ of $X_\Gamma$ of measure at least $2/3$ on which for all $u,v\in C_n$, 
\[d_{S_\Gamma}(b_x(v),b_x(u))\ge \frac{r(n)}{n}d_{S_\Lambda}(u,v).\]
	Finally, for all $x$ in the subset  $A\cap B$ which has positive measure, we deduce for every $u,v\in C_n$ that
	\[a_nd_{C_n}(u,v)\le d_{S_\Gamma}(b_x(v),b_x(u)) \le b_nd_{C_n}(u,v),\]
	where in the case of \cref{thm:hyperbolicExpME}, $a_n=\frac{r(n)}{18 n}$ and $b_n=\psi^{-1}\left(3Ln\right)$; 
	and in the case of~\ref{thm:hyperbolicExpMELinfty}, $a_n=\frac{r(n)}{18 n}$ and $b_n=L$. 
In the first case, assuming (\ref{eq:cond2"}), we have 
\[\frac{r(n)}{18}\geq 4\delta \log_2 n+ 3\psi^{-1}\left(3Ln\right),\]
from which we deduce that for $n$ large enough (as $b_n\to \infty$),
	\[na_n\geq  4\delta \log_2 n +3b_n >4\delta \log_2 (b_nn)+4+2b_n,\]
	which contradicts \cref{prop:HypNotEmbed}.
	The second case is similar: assuming (\ref{eq:cond2'''}), we get
 \[na_n=\frac{r(n)}{18}\geq 6\delta \log n,\] 
which for $n$ large enough contradicts \cref{cor:HypNotEmbed}. 
\end{proof}

\bibliographystyle{alphaurl}
\bibliography{bib}

\begin{thebibliography}{DKLMT22}

\bibitem[Aus16]{austinIntegrableMeasureEquivalence2016}
Tim Austin.
\newblock Integrable measure equivalence for groups of polynomial growth.
\newblock {\em Groups, Geometry, and Dynamics}, 10(1):117--154, 2016.
\newblock \href {https://doi.org/10.4171/GGD/345} {\path{doi:10.4171/GGD/345}}.

\bibitem[Bel68]{belinskayaPartitionsLebesguespace1968}
R.~M. Belinskaya.
\newblock Partitions of {{Lebesgue}} space in trajectories defined by ergodic
  automorphisms.
\newblock {\em Functional Analysis and Its Applications}, 2(3):190--199, 1968.
\newblock \href {https://doi.org/10.1007/BF01076120}
  {\path{doi:10.1007/BF01076120}}.

\bibitem[BFS13]{baderIntegrableMeasureEquivalence2013}
Uri Bader, Alex Furman, and Roman Sauer.
\newblock Integrable measure equivalence and rigidity of hyperbolic lattices.
\newblock {\em Inventiones mathematicae}, 194(2):313--379, 2013.
\newblock \href {https://doi.org/10.1007/s00222-012-0445-9}
  {\path{doi:10.1007/s00222-012-0445-9}}.

\bibitem[Bow17]{Bowen2017}
Lewis Bowen.
\newblock Integrable orbit equivalence rigidity for free groups.
\newblock {\em Israel Journal of Mathematics}, 221:471--480, 2017.
\newblock \href {https://doi.org/10.1007/s11856-017-1553-4}
  {\path{doi:10.1007/s11856-017-1553-4}}.

\bibitem[{de }19]{delasalleStrongPropertyHigherrank2019}
Mikael {de la Salle}.
\newblock Strong property ({{T}}) for higher-rank lattices.
\newblock {\em Acta Mathematica}, 223(1):151--193, September 2019.
\newblock \href {https://doi.org/10.4310/ACTA.2019.v223.n1.a3}
  {\path{doi:10.4310/ACTA.2019.v223.n1.a3}}.

\bibitem[DK18]{drutuGeometricGroupTheory2018}
Cornelia Dru{\c t}u and Michael Kapovich.
\newblock {\em Geometric Group Theory. {{With}} an Appendix by {{Bogdan
  Nica}}}, volume~63 of {\em Colloq. {{Publ}}., {{Am}}. {{Math}}. {{Soc}}.}
\newblock {American Mathematical Society (AMS)}, {Providence, RI}, 2018.
\newblock \href {https://doi.org/10.1090/coll/063}
  {\path{doi:10.1090/coll/063}}.

\bibitem[DKLMT22]{delabieQuantitativeMeasureEquivalence2022}
Thiebout Delabie, Juhani Koivisto, Fran{\c c}ois Le~Ma{\^i}tre, and Romain
  Tessera.
\newblock Quantitative measure equivalence between amenable groups.
\newblock {\em Annales Henri Lebesgue}, 5:1417--1487, 2022.
\newblock \href {https://doi.org/10.5802/ahl.155} {\path{doi:10.5802/ahl.155}}.

\bibitem[Fur99]{furmanGromovMeasureEquivalence1999}
Alex Furman.
\newblock Gromov's {{Measure Equivalence}} and {{Rigidity}} of {{Higher Rank
  Lattices}}.
\newblock {\em Annals of Mathematics}, 150(3):1059--1081, 1999.
\newblock \href {https://doi.org/10.2307/121062} {\path{doi:10.2307/121062}}.

\bibitem[Gro93]{gromovAsymptoticInvariantsInfinite1993}
Mikhael Gromov.
\newblock Asymptotic invariants of infinite groups.
\newblock In {\em Geometric Group Theory: Proceedings of the {{Symposium}} Held
  in {{Sussex}}, 1991: Volume 2}, London {{Mathematical Society}} Lecture Note
  Series: 182. {Cambridge University Press}, 1993.

\bibitem[Gro07]{gromovMetricStructuresRiemannian2007}
Mikhael Gromov.
\newblock {\em Metric Structures for {{Riemannian}} and Non-{{Riemannian}}
  Spaces}.
\newblock Modern {{Birkh\"auser}} Classics. {Birkh\"auser}, {Boston}, 2007.

\bibitem[HM20]{humePoorlyConnectedGroups2020}
David Hume and John Mackay.
\newblock Poorly connected groups.
\newblock {\em Proceedings of the American Mathematical Society},
  148(11):4653--4664, 2020.
\newblock \href {https://doi.org/10.1090/proc/15128}
  {\path{doi:10.1090/proc/15128}}.

\bibitem[OW80]{ornsteinErgodicTheoryAmenable1980}
Donald~S. Ornstein and Benjamin Weiss.
\newblock Ergodic theory of amenable group actions. {{I}}. {{The Rohlin}}
  lemma.
\newblock {\em Bull. Amer. Math. Soc. (N.S.)}, 2(1):161--164, 1980.
\newblock \href {https://doi.org/10.1090/S0273-0979-1980-14702-3}
  {\path{doi:10.1090/S0273-0979-1980-14702-3}}.

\bibitem[Sau02]{sauerMathrmLInvariantsGroups2002}
Roman Sauer.
\newblock {\em $\mathrm{L}^2$-{{Invariants}} of {{Groups}} and {{Discrete
  Measured Groupoids}}}.
\newblock PhD thesis, Universit\"at M\"unster, 2002.

\bibitem[Sha00]{shalomRigidityUnitaryRepresentations2000}
Yehuda Shalom.
\newblock Rigidity, {{Unitary Representations}} of {{Semisimple Groups}}, and
  {{Fundamental Groups}} of {{Manifolds}} with {{Rank One Transformation
  Group}}.
\newblock {\em Annals of Mathematics}, 152(1):113--182, 2000.
\newblock \href {https://doi.org/10.2307/2661380} {\path{doi:10.2307/2661380}}.

\bibitem[Sha04]{shalomHarmonicAnalysisCohomology2004}
Yehuda Shalom.
\newblock Harmonic analysis, cohomology, and the large-scale geometry of
  amenable groups.
\newblock {\em Acta Math.}, 192(2):119--185, 2004.
\newblock \href {https://doi.org/10.1007/BF02392739}
  {\path{doi:10.1007/BF02392739}}.

\bibitem[VS14]{verbeekMetricEmbeddingHyperbolic2014}
Kevin Verbeek and Subhash Suri.
\newblock Metric {{Embedding}}, {{Hyperbolic Space}}, and {{Social Networks}}.
\newblock In {\em Proceedings of the {{Thirtieth Annual Symposium}} on
  {{Computational Geometry}}}, {{SOCG}}'14, pages 501:501--501:510, {New York,
  NY, USA}, 2014. {ACM}.
\newblock \href {https://doi.org/10.1145/2582112.2582139}
  {\path{doi:10.1145/2582112.2582139}}.

\end{thebibliography}
\listoffixmes
\end{document}